\documentclass[11pt]{article}

\usepackage{amsfonts}
\usepackage{amsmath}
\usepackage{amsthm}
\usepackage{color}
\usepackage{comment}
\usepackage{lineno}
\usepackage{graphicx}

\newtheorem{theorem}{Theorem}
\newtheorem{lemma}[theorem]{Lemma}
\newtheorem{conj}{Conjecture}
\newtheorem{question}[conj]{Question}
\newtheorem{corollary}[theorem]{Corollary}
\newtheorem*{defn}{Definition}
\theoremstyle{remark}
\newtheorem*{remark}{Remark}

\def\a{\alpha}
\def\b{\beta}
\def\e{\varepsilon}
\DeclareMathOperator\var{Var}
\newcommand\rtot{R_{{\rm tot}}}
\def\eps{\varepsilon}

\def\E{\mathbb{E}}
\def\N{\mathbb{N}}
\def\R{\mathbb{R}}
\def\cB{\mathcal{B}}
\newcommand\Prb{\mathbb{P}}

\begin{document}

\title{Optimal Resistor Networks}
\author{J.~Robert Johnson\footnote{School of Mathematical Sciences, Queen Mary,
University of London, London E1 4NS, England. e-mail: \tt{r.johnson@qmul.ac.uk}} \and Mark Walters\footnote{School of Mathematical Sciences, Queen Mary,
University of London, London E1 4NS, England. e-mail: \tt{m.walters@qmul.ac.uk}}}
\date\today
\maketitle
\begin{abstract}

Given a graph on $n$ vertices with $m$ edges, each of unit resistance,
how small can the average resistance between pairs of vertices be?
There are two very plausible extremal constructions -- graphs like a
star, and graphs which are close to regular -- with the transition
between them occuring when the average degree is $3$. However, one of
our main aims in this paper is to show that there are significantly
better constructions for a range of average degree including average degree near~$3$.

  A key idea is to link this question to a analogous question about
  rooted graphs -- namely `which rooted graph minimises the average
  resistance to the root?'. The rooted case is much simpler to analyse
  that the unrooted, and one of the main results of this paper is that
  the two cases are asymptotically equivalent.
\end{abstract}

\section{Introduction}\label{s:intro}

In this paper we shall be considering resistor networks. For our
purposes a resistor network is a graph $G$ (possibly containing
multiple edges) where we view each edge as having unit
resistance. (Note, for the reader unfamiliar with the theory of
electrical networks, we give a brief summary of the facts we need in
Appendix~A; alternatively see Snell and Doyle~\cite{MR920811} or
Chapters~2 and~9 of Bollob\'as~\cite{MR1633290}.)  For convenience we
assume that $G$ has vertex set $V=[n]$ and edge set $E$, and for any such
graph $G$, we define the average resistance to be
\[
A(G)=\frac{1}{\binom{n}{2}}\sum_{\substack{x,y\in V\\ x<y}}R_{xy},
\]
where $R_{xy}$ denotes the effective resistance between $x$ and $y$ in
$G$.  We would like to know which graphs minimise this quantity for a
given number $n$ of vertices and $m$ of edges.  Trivially,
if $G$ is not connected then $A(G)=\infty$, so we shall always assume
$m\ge n-1$.

In the next section we describe two very plausible optimal graphs: one
is roughly `star-like' (i.e., contains a vertex of very high degree),
and the other is `regular-like'. It is easy to verify that for
$m<3n/2$ the star-like constructions are better than the regular-like
constructions, whereas the reverse is true when $m>3n/2$, with both
types of graph obtaining $A(G)=4/3$ when $m=3n/2$. Thus it seems
natural to expect that the optimal graph starts out star-like and
switches to regular (or regular-like) when $m=3n/2$.  However,
contrary to our expectations, this turns out not to be the case: there
is a signifcantly better construction for $m=3n/2$.
\begin{theorem}\label{t:alpha=3}
  For all sufficiently large $n$ there is a graph $G$ with $n$
  vertices and at most $3n/2$ edges and $A(G)<1.2908$.
\end{theorem}
\noindent%
Describing this construction and proving an analogue of
Theorem~\ref{t:alpha=3} for a range of $m$ is one of the main aims of
this paper.

One key idea of the proof is to consider a related model where we 
insist that all current go `via' a special designated `root' vertex. More
precisely, we consider graphs with a root $\rho$ and let $B(G)$ be the
average resistance of a vertex to the root.

It is simple to prove (see Lemma~\ref{l:triangle-inequality} in Appendix~A) that
$R_{xy}\le R_{x\rho}+R_{\rho y}$ and, hence, that $A(G)\le 2B(G)$.
It follows that graphs giving upper bounds on $B(G)$ automatically give upper bounds
for $A(G)$. Our second main result shows that, by slightly modifying
the graph and allowing a small error term, we can obtain the reverse
inequality.
\begin{theorem}\label{root-unroot:t}
  Suppose that $G$ is a graph and that $s\in \N$ is given. Then there
  is a graph $G'$ formed by adding a root vertex to $G$ joined to $G$
  by $s$ edges with
  \[
  B(G')\le \frac12 A(G) + \frac1{2s} A(G)+\frac{1}{s}.
  \]
\end{theorem}
\noindent%
The star shows that any optimal graph has $A(G)\le 2$, so
Theorem~\ref{root-unroot:t} shows that, for any such graph $G$, there
is a graph $G'$ such that $B(G')\le A(G)/2+2/s$.  Letting $s$ tend to
infinity with $n$ shows that the two models are equivalent
asymptotically.  Thus, it suffices to prove all our bounds for the
rooted model. This is useful as the rooted model is often much simpler
to analyse.

\subsubsection*{Motivation}

In addition to being mathematically natural, the question of
minimising $A(G)$ arises in
statistics~\cite{MR2409757,MR1105858}. Indeed, suppose that $G$ is the
concurrence graph of an experimental design. Then $R_{xy}$ is
proportional to the variance of the best linear unbiased estimator for
the difference between treatments $x$ and $y$. It follows that the
graphs for which $A(G)$ is minimised correspond to designs for which
the average of these variances is as small as possible. See the
surveys \cite{MR2588537,MR3026766} for more details. We note that
whilst our primary interest in the rooted model is to allow us to
prove results about the general model, it does, in fact, have a
natural interpretation in this application. A statistician would call
the rooted model a `design with a control.'

The parameter $A(G)$ has also been studied in the context of
mathematical chemistry~\cite{QUA:QUA1} and network
robustness~\cite{MR2811133}. In these areas it is sometimes referred
to as the Kirchhoff Index of a graph.

Finally, resistance plays an important role in the theory of random
walks on graphs. This connection was first shown in Snell and
Doyle~\cite{MR920811}, and has been further developed by many authors
including Chandra et al.~\cite{MR1613611}, Coppersmith, Feige and
Shearer~\cite{MR1386885}, and Tetali~\cite{MR1088395}.

\subsubsection*{Layout of the Paper}
Before discussing the main layout of the paper we note that, for the
reader unfamiliar with the theory of electrical networks, we include
all the definitions and results we need in Appendix~A.

We start by formalising both the unrooted and rooted models and
introduce some notation. Then in Section~\ref{s:easy-facts} we use
some basic ideas from electrical network theory to give some simple
general bounds. These bounds justify our choice of $m$ linear in $n$
as being the most interesting case of the problem.

In Section~\ref{s:basic} we start the main paper by proving some
elementary results about the rooted average resistance and, in
Section~\ref{section:equivalence}, we prove that the rooted and unrooted
models have the same asymptotic behaviour.

Having concluded the necessary preparation we prove our upper bounds
including Theorem~\ref{t:alpha=3} and a proof of a bound given by a
random regular graph heuristic in Section~\ref{s:upper}; and improve
our lower bounds in Section~\ref{s:lower}.

In Section~\ref{s:optimal-graphs} we show how the stucture changes as
$m$ increases -- in broad terms, from graphs which are like the star
to graphs which are closer to a regular graph.  In
Section~\ref{s:queen-bee} we discuss the rooted model with the extra
requirement that the root must be joined to all vertices. This model
is sometimes called the \emph{Queen-Bee} model by statisticians (see
e.g.,~\cite{MR3026766}). We show that, for any $n\le m\le 3n/2$, the
best Queen-Bee network is exactly a star of triangles and
leaves. However, we will see that no Queen-Bee network is even
asymptotically optimal unless $m=(1+o(1))n$.

We conclude with some open problems.

\section{Definitions and
  Notation}\label{s:defns} Recall from the introduction that we define
\[
A(G)=\frac{1}{\binom{n}{2}}\sum_{\substack{x,y\in V\\ x<y}}R_{xy},
\]
to be the average resistance between pairs of vertices. However, in
most of our calculations it turns out to be easier to consider the
average over all pairs, rather than all distinct pairs. Hence, we define
\[
A'(G)=\frac{1}{n^2}\sum_{x,y\in V(G)}R_{xy}. 
\]
Since $R_{xx}=0$ for all $x$ we have
\[A(G)=\frac{n^2}{n(n-1)}A'(G)=\frac{n}{n-1}A'(G)
\]
so, in particular, the asymptotic behaviour is the same for $A$ and
$A'$.

For the rooted model we define
\[
B(G)=\frac{1}{|V(G)|-1}\sum_{x\in V(G)}R_{\rho x}
\]
to be the average resistance to the root vertex $\rho$.

We are interested in minimising these two quantities for a given size
of graph. As we shall see in the next section, it is simpler to
parameterise by the average degree of the graph rather than the number
of edges. Thus define
\[
a_n(\a)=\min\{A'(G) : |V(G)|=n,\ e(G)\le \frac{\alpha n}{2}\}
\]
and 
\[
b_n(\a)=\min\{B(G) : |V(G)|=n+1,\ e(G)\le \frac{\alpha n}{2}\}.
\]
Requiring $|V(G)|=n+1$ in the second definition makes the calculations
a little cleaner: it means that $n$ is the number of non-root
vertices. With a slight abuse of terminology we will refer to $\alpha$
as the average degree in both cases.

As most of our work is with the rooted model it is convenient
to define some extra notation. The
\emph{total effective resistance} $\rtot=\sum_{x\in V(G)}R_{\rho
  x}$. In some cases we will want to refer to the graphs occuring in
the definition of $b_n(\a)$. Thus we define
\[
\cB(n,\a)=\{G : |V(G)|=n+1,\ e(G)\le \frac{\alpha n}{2}\}.
\]

Many of our results will be asymptotic and we define
\[a(\a)=\lim_{n\to\infty}a_n(\a)\qquad\text{and}\qquad
b(\a)=\lim_{n\to\infty}b_n(\a).\] We prove that these limits exist in
the next two sections.

\section{Background and Elementary Remarks}\label{s:easy-facts}
In this section we discuss some easy cases of our problem, and give
some simple general bounds. The results in this section are not
original (see~\cite{MR1105858} or~\cite{MR2409757} for
example) and are given to place our main result
(Theorem~\ref{t:alpha=3}) into context.

As we remarked in the introduction, if $G$ is not connected then
$A(G)$ is infinite.  Thus, the first non-trivial case is $m=n-1$ and
then any $G$ minimising $A(G)$ must be a tree. Since a tree contains a
unique path between any two points, the resistance between any two
vertices is just the distance between them. This implies that a star
is the unique (up to isomorphism) graph minimising $A(G)$.  In this
case, the average resistance is $2-2/n$.

In the next case $n=m$ the situation is a little more complicated. As
above $G$ must be connected so it must be unicyclic. It is easy to
check that the graph minimising $A(G)$ must consist of a cycle with
(possibly) some leaves attached to one vertex of the cycle. Then it is a simple 
calculation to find the optimal length for the cycle: it turns out
that for $n\le 13$ a 4-cycle is optimal whereas for $n\ge 13$ a
3-cycle is optimal (they are both optimal for $n=13$); see
\cite{MR2409757} for details.

The fact that the optimal graph switches as $n$ increases leads us to
our first choice: in this paper we will almost always be looking at
the behaviour for `large $n$'. Thus, we suppose that $n$ is large, and
that $G$ has average degree $\alpha=\alpha(n)$ so $m=n
\alpha/2$. Although we have allowed $\alpha$ to vary with $n$, as we
shall see shortly, the case where $\alpha$ is constant is the most
interesting.

If $\alpha$ is an even integer then one possible graph is a star where
each vertex is joined to the centre of the star by $\alpha/2$ edges.
This gives a graph with average degree approximately $\alpha$ and
average resistance approximately $4/\alpha$.

What about lower bounds?  The first trivial bound we might consider is
that, for any $x\not=y$, $R_{xy}\ge 1/d_x$ where $d_x$ denotes the
degree of $x$. Thus, by convexity,
\begin{equation}\label{trivial-average:e}
A(G)\ge \frac1{n}\sum_{x\in V} \frac{1}{d_x}\ge \frac{1}{(\frac 1n\sum_xd_x)}=\frac{1}{\alpha},
\end{equation}
i.e., $A(G)$ is at least the reciprocal of the average
degree. Combining this lower bound with the upper bound given by the
multi-edge-star construction above we see that, as the average degree $\alpha$ tends to
infinity, the average resistance for an extremal graph is $\Theta(1/\alpha)$.

These upper and lower bounds are approximately a factor of four apart;
it is easy, however, to improve the lower bound in
(\ref{trivial-average:e}). For all pairs $x$, $y$ which are not
neighbours
\[R_{xy}\ge \frac{1}{d_x}+\frac{1}{d_y}.
\]
Thus, since $d_x\ge 1$ for all $x$,
\begin{align*}
A(G)&\ge \frac1{n(n-1)}\sum_{\substack{x,y\\x\not=y\\y \in V\setminus \Gamma(x)}} \left(\frac{1}{d_x}+\frac{1}{d_y}\right)\\
&\ge \frac1{n(n-1)}\sum_{\substack{x,y\\x\not=y}}\left( \frac{1}{d_x}+\frac{1}{d_y}\right) -\frac{4m}{n(n-1)}\\
&\ge \frac{2}{(\frac 1n\sum_xd_x)}-O\left(\frac{m}{n(n-1)}\right)\\
&= \frac{2}{\alpha}-O\left(\frac{\alpha}{n}\right).
\end{align*}
In other words, provided the average degree $\alpha$ is $o(n)$ the
average resistance cannot be (much) less than $2/\alpha$.

Another approach for a lower bound is to prove a lower bound for the
`continuous' version of the model where non-integral resistances are
allowed: that is $G$ is a complete graph with arbitrary resistances on
the edges. The natural analogue of the bound on the number of edges is
to bound the total \emph{conductance}, that is $\sum_{uv\in E(G)}1/r_{uv}\le m$,
where $r_{uv}$ denotes the resistance of the edge $uv$.

It is well known that (see Appendix~A), given two networks
$G_1$ and $G_2$ with total conductance at most $m$, the network
$H=\frac12(G_1+G_2)$ (i.e. with the conductance of an edge the average
of the conductance in the two networks) has $A(H)\le
\frac12\left(A(G_1)+A(G_2)\right)$. This implies that the network with
all conductances equal is an optimal network. In this network each edge has
conductance $m/\binom n2$, and it easy to check that average
resistance is $(n-1)/m$, which is approximately $2/\alpha$.

  Having found some simple lower bounds let us give some more
  constructions; we concentrate on the case of small $\alpha$. The
  above averaging argument indicates that, in a sense, the star
  construction is the worst possible: we have `averaged' as little as
  possible. Thus let us look at the opposite extreme: a very uniform
  graph. Suppose that $G$ is a typical $\alpha$-regular graph. If we
  assume that the graph locally looks like a tree then we might expect
  (as a heuristic) the resistance between two points to be twice the
  resistance of a vertex to infinity along an $\alpha$-regular
  tree. By identifying vertices at the same depth in the tree (since they will all
  have the same potential in the flow from the vertex to infinity) it
  is easy to see that the resistance in an $\alpha$-regular tree is
\[
\sum_{k=0}^\infty\frac1{\alpha(\alpha-1)^k}=\frac{(\alpha-1)}{\alpha(\alpha-2)},
\]
which suggests an average resistance between vertices of
$\frac{2(\alpha-1)}{\alpha(\alpha-2)}$.  

If we compare this with the lower bound of $2/\alpha$ we proved above,
we see that if $\alpha\to \infty$ then our upper and lower bounds are
both $\sim 2/\alpha$. (We note, though, that the upper bound as we
have given it here is just a heuristic; we discuss this more in
Section~\ref{s:upper}.) However, for $\alpha$ a fixed constant the
upper and lower bounds do differ significantly.

Let us compare this new upper bound with the upper bound from the
`star construction'. We see that the regular construction is better
than the star construction for $\alpha \ge 4$. For $\alpha=3$ this
heuristic gives average resistance $4/3$ which is the same as the
bound for the star when $\alpha=3$. The bound for the star was only
valid for $\alpha$ an even integer; however there is another star-like
construction: a star of triangles. In this construction one vertex is
joined to all others, and a matching is added between all the 
non-centre vertices (if $n$ is even there is one vertex left over,
which we can ignore as we are looking at the asymptotics). This does
have average resistance asymptotically $4/3$. Thus, for $\alpha=3$ the
`star-like' construction and the random regular heuristic give the
same average resistance.

Finally, we note that the observations we have given so far naturally
motivate the rooted model we defined in the introduction.  Indeed, let
us consider the case $m=n$ again: we saw above that, depending on $n$,
the optimal graph was either a 4-cycle with leaves, or a 3-cycle with
leaves. The reason that the optimal graph switches is that, for small
$n$, the resistances between pairs of vertices in the cycle form a
significant part of the average, whereas for large $n$ the average is
dominated by leaf vertex to leaf vertex resistance (always 2) followed
by the contribution from leaf vertices to cycle vertices.  The rooted
model naturally avoids this change in behaviour as it insists that all
current go to the root vertex; so in the example above the optimal
rooted graph is always a 3-cycle with leaves.

\section{Basic Properties of the Rooted Model}\label{s:basic}
In this section we work entirely with the rooted model; in the next
section we show that the rooted model and the unrooted model behave
the same asymptotically, so most of the results proved in this section
also apply to the unrooted model. It is worth remarking, however, that
we do not have direct proofs of most of the analogues of these results
for the unrooted model: the only proof we know is via the results of
this section and the equivalence between the two models given in the
next section. 

We first show that the limit in the definition of the function $b$
exists.

\begin{lemma}\label{limit:lem}
For any $\alpha\ge 2$, the limit $b(\a)=\lim_{n\to\infty}b_n(\a)$ exists and is equal to $\inf_n b_n(\a)$.
\end{lemma}
\begin{proof}
  Given $\a\geq2$ let $b=\inf_nb_n(\a)$. Given $\e>0$ choose $n$ so
  that $b_n(\a)<b+\e$ and let $G\in\cB(n,\a)$ be a rooted
  graph with $B(G)=b_n(\a)$. 

  Any $N$ can be written as $N=kn+l$ with $0\leq l<n$. Consider the
  graph~$H$ on~$N+1$ vertices comprising $k$ copies of $G$ with a
  common root vertex, but otherwise disjoint, together with $l$ leaves
  adjacent to the root. This graph has
  $e(H)=ke(G)+l\leq\frac{k\a n}{2}+l\leq\frac{\a N}{2}$ and so
  $H\in\cB(N,\a)$.  The average rooted resistance is
  $B(H)=\frac{l}{N}+\left(1-\frac{l}{N}\right)B(G)<\frac{n}{N}+b+\e$. Therefore,
  for any $N>\frac{n}{\e}$ we have $B(H)<b+2\e$. Thus,
  $\limsup_{N\to\infty} b_N(\a)<b+2\e$ and the result follows.
\end{proof}

\begin{lemma}\label{convex:lem}
The function $b$ is convex and continuous on $[2,\infty)$.
\end{lemma}
\begin{proof}
Suppose that $2\leq\a\leq\b$, that $t\in[0,1]$ and that
$\gamma=t\a+(1-t)\b$. We must show that $b(\gamma)\leq
tb(\a)+(1-t)b(\b)$.

Given $\e>0$ take $n$ large enough that $b_n(\a)<b(\a)+\e$ and
$b_n(\b)<b(\b)+\e$. Let $G_1\in\cB(n,\a),G_2\in\cB(n,\b)$ be rooted graphs with
$B(G_1)=b_n(\a)$ and $B(G_2)=b_n(\b)$.

For any integer $k$ let $N_k=kn$ and let $H$ be the graph comprising
$\lceil tk\rceil$ copies of $G_1$ and $\lfloor(1-t)k\rfloor$ copies of
$G_2$ with a common root vertex but otherwise disjoint. The graph $H$
has $(\lceil tk\rceil+\lfloor(1-t)k\rfloor)n+1=N_k+1$ vertices and
$e(H)\leq\frac{\gamma N_k}{2}$; i.e., $H\in \cB(N_k,\gamma)$. Also
\begin{align*}
B(H)&=\frac{\lceil
  tk\rceil}{k}B(G_1)+\frac{\lfloor(1-t)k\rfloor}{k}B(G_2)\\ 
&<(t+\e)(b(\a)+\e)+(1-t)(b(\b)+\e)
\end{align*}
provided that $k$ is large. This expression tends to
$tb(\a)+(1-t)b(\b)$ as $\e\rightarrow0$ and so $b(\gamma)\leq
tb(\a)+(1-t)b(\b)$ as required.

This shows that $b$ is convex on $[2,\infty)$ which implies it is
continuous on $(2,\infty)$. To prove that $b$ is continuous at $2$ we
need to use the lower bound of $1/(\a-1)$ for $b(\a)$ that we prove
later in Theorem~\ref{t:2-step-lower}. Since this bound is continuous
and takes the value $1$ when $\a=2$ this completes the proof.
\end{proof}
\begin{remark}
  This result is only true in the limit: that is $b$ is convex but,
  for fixed $n$ the function $b_n$ is not. Indeed, we have seen that
  $b_n(n-1)=1$ and $b_n(n)=1-\frac{2}{3n}$. On the other hand, by
  Theorem~\ref{t:alpha=3} combined with the equivalence between the
  rooted and unrooted models (Theorem~\ref{root=unroot:t}, see later),
  we see that $b_n(\frac{3n}{2})\le 0.65<\frac23$ for all sufficently large
  $n$.
\end{remark}

Although we are interested in minimising the average resistance, it
will sometimes be useful to have a bound on the maximum resistance in
an optimal graph. The following simple lemma provides this.

\begin{lemma}\label{contract:lem}
  If $G\in\cB(n,\a)$ and $R_{x\rho}>1$ for some $x\in V(G)$ then there is a
  graph $H\in\cB(n,\a)$ with $B(H)<B(G)$. 
\end{lemma}
 
\begin{proof}
  Suppose that $R_{x\rho}>1$. Contracting an edge incident with $x$
  gives a graph $G'$ with $n-1$ non-root vertices and at most $e(G)-1$
  edges (exactly $e(G)-1$ unless the contracted edge was one of a set
  of parallel edges). We have not increased the resistance between any
  of the remaining vertices and the root so
  $\rtot(G')<\rtot(G)-1$. Now adding a leaf to the root gives a graph
  with~$n$ non-root vertices and at most $e(G)$ edges. The new vertex
  has resistance 1 to the root and so $\rtot(H)\leq
  \rtot(G')+1<\rtot(G)$. Since $H$ has the same number of vertices as
  $G$ and at most as many edges this shows that $B(H)<B(G)$ as
  claimed.
\end{proof}

\section{Equivalence of Models}\label{section:equivalence}

In this section we prove Theorem~\ref{root-unroot:t}
showing that the rooted and unrooted models behave the same
asymptotically. As a corollary we deduce that many of the results of
the previous section also hold in the unrooted case.

\begin{defn}
  For any graph $G$ and multiset $S$ of vertices let $G_S$ be the
  graph formed from $G$ by adding a new vertex $\rho$ joined to each
  vertex of $S$ (by multiple edges if the vertex occurs multiple
  times). We call the new vertex the \emph{root} and the vertices in
  $S$ the \emph{sinks}.
\end{defn}
\begin{proof}[Proof of Theorem~\ref{root-unroot:t}]
  In fact we show that there is a multiset $S$ of vertices with $|S|=n$
  such that the graph $G_S$ satisfies
  \[
  B(G_S)\le \frac12 A'(G) + \frac1{2s} A'(G)+\frac{1}{s}.
  \]
  Since $A'(G)=\frac{n-1}{n}A(G)$ this implies the theorem with
  $G'=G_S$.

  We do our calculations in $G$. Fix one vertex $v\in G$ and let
  $I_i(e)$ be the current in edge $e$ when a current of size 1 flows into 
  the network at $i$ and out from the network at $v$. By the principle of
  superposition the current in the edge $e$ when a current of size 1
  flows in at $i$ and out at $j$ is $I_i(e)-I_j(e)$.

  The total power of the flow when a current of size 1 enters at
  $i$ and leaves at $j$ equals the resistance $R_{ij}$. Thus
  \[
  R_{ij}=\sum_e (I_j(e)-I_i(e))^2
  \]
  and, hence, the average resistance over all pairs is 
  \begin{equation}\label{e:sum-ij-resistance}
  \frac{1}{n^2} \sum_{i,j} R_{ij}=\frac{1}{n^2}\sum_{i,j}\sum_e
  (I_j(e)-I_i(e))^2=\sum_e \frac{1}{n^2}\sum_{i,j}(I_j(e)-I_i(e))^2.
  \end{equation}

  For each $e$ let $X(e)$ be the random variable obtained by picking a
  vertex $i$ uniformly at random and setting $X(e)=I_i(e)$ and let
  $X'(e)$ be an independent copy of $X(e)$. Then equation
  (\ref{e:sum-ij-resistance}) becomes
  \begin{align*}
    A'(G)=\frac{1}{n^2} \sum_{i,j} R_{ij}
    &=\sum_e \var(X(e)-X'(e))&\text{since $\E(X(e)-X'(e))=0$}\\
    &=\sum_e \var(X(e))+\var(X'(e))&\text{since $X$ and $X'$ are independent}\\
    &=2\sum_e\var(X(e)).
  \end{align*}

  Next we bound $B(G_S)$. Suppose that $S$ is a fixed multiset of $s$
  vertices. One flow of size 1 entering at $i$ and leaving at $\rho$
  is given by sending a current of $1/s$ to each of the sinks and on
  to the root: that is the flow given by $\frac1s\sum_{j\in S}
  (I_i-I_j)$, together with a current of size $1/s$ in each of the
  edges to the root. We know that the actual current flow is the one
  minimising the power, so the actual power is at most the power in
  this flow. Thus
  \[
  R_{i\rho}\le \sum_e \left(\frac1s\sum_{j\in S}(I_i(e)-I_j(e))\right)^2+s\cdot\frac1{s^2}.
  \]

  Averaging this  over all $i$
  gives
  \begin{align*}\label{e:rbarw}
    B(G_S)
    &\le \frac{1}{n}\sum_{i\in V} \sum_e \left(\frac1s\sum_{j\in
        S}(I_i(e)-I_j(e))\right)^2+\frac1{s}\notag\\
    &\le \frac{1}{n}\sum_{i\in V} \sum_e \left(I_i(e)-\frac1s\sum_{j\in
        S}I_j(e)\right)^2+\frac1{s}\notag\\
    &= \E_X\left(\sum_e \left(X(e)-\frac1s\sum_{j\in
        S}I_j(e)\right)^2\right)+\frac1{s}\notag\\
    &= \sum_e \E_X \left(\left(X(e)-\frac1s\sum_{j\in
        S}I_j(e)\right)^2\right)+\frac1{s}\notag\\
  \end{align*}
  
  Now, suppose that we pick $S$ uniformly at random from all
  multisets of $s$ vertices (i.e., a random $s$-tuple of
  vertices), and independently of $X$ and $X'$. Obviously
  \[\E_S\left(\frac1s\sum_{j\in S}I_j(e)\right)=\E_X(X(e))\qquad\text{so}\qquad
  \E_{X,S} \left(X(e)-\frac1s\sum_{j\in
        S}I_j(e)\right)=0
  \]
Hence, we get
\begin{align*}
    \E_S(B(G_S))&\le  \sum_e \E_{X,S} \left(\left(X(e)-\frac1s\sum_{j\in
        S}I_j(e)\right)^2\right)+\frac1{s}\\
  &=\sum_e \var_{X,S} \left(X(e)-\frac1s\sum_{j\in
        S}I_j(e)\right) +\frac1{s}\\
  &=\sum_e \var_{X} \left(X(e)\right)+\sum_e\var_{S}\left(\frac1s\sum_{j\in
        S}I_j(e)\right) +\frac1{s}\\
  &=\sum_e \var_{X} \left(X(e)\right)+\frac1s\sum_e\var_{X}(X(e)) +\frac1{s}\\
  &=\frac{A'(G)}{2}+\frac{A'(G)}{2s} +\frac1{s}.
\end{align*}
Thus the result is true `on average' so there must be a set $S$ for which it holds.
\end{proof}
We remark that we could insist that $S$ is actually a set (rather
than a multiset). Indeed, standard results for the variance of a sum
with, and without, replacement show that
\[\var_{S}\left(\frac1s\sum_{j\in
    S}I_j(e)\right)\] is smaller for sets than multisets, and using
this in the estimation of $\E_S(B(G_S))$ implies this slightly
stronger result.

\begin{theorem}\label{root=unroot:t}
  For any $\alpha$ we have $b(\alpha)=a(\alpha)/2$.
\end{theorem}
\begin{proof}
  First observe that $a(\alpha)\le 2b(\alpha)$. Indeed, for any rooted
  graph $G$ the triangle inequality for resistances
  (Lemma~\ref{l:triangle-inequality}) gives $R_{xy}\le R_{x\rho}+R_{\rho y}$
  which implies $A(G)\le 2B(G)$.

  For the reverse inequality suppose that $G_n$ is a sequence of (unrooted)
  graphs with $n\to \infty$, $e(G_n)\le \alpha n/2$ and $A(G_n)\to
  a(\alpha)$. Let $s=\log n$ and form graphs $H_n$ as given by
  Theorem~\ref{root-unroot:t}. The graph $H_n$ has $n+1$ vertices, at
  most $\alpha n +\log n$ edges, and resistance to the root at most
  $A'(G)/2+O(1/\log n)$. Thus
\[
\lim_{x\to \alpha^+}b(x)\le a(\alpha)/2.
\]
Since $b$ is continuous (Lemma~\ref{convex:lem}) the result follows.
\end{proof}
Theorem~\ref{root=unroot:t} and Lemmas~\ref{limit:lem}
and~\ref{convex:lem} combine to prove
\begin{corollary}
  The function $a(\a)$ exists and is convex and continuous on~$[2,\infty)$.\qed
\end{corollary}
As we remarked previously we do not know of any direct proof of the
convexity of $a$ without going via the equivalence with $b$.

\section{Constructions}\label{s:upper}

\subsection{Stars of Edges and Triangles}

We start with a simple construction for $\alpha$ between 2 and 3.
\begin{lemma}\label{star:lem}
For $\a\in[2,3]$ we have  $b(\a)\leq\frac{5-\a}{3}$.
\end{lemma}
\begin{proof}
  For $\a=3$ and $n$ even, let $G$ be a star of triangles: that is a
  collection of triangles each containing the root but otherwise
  disjoint. It is easy to check that the resistance of any vertex to
  the root is $2/3$ and hence $b_n(3)\le 2/3$. Since $b_n(2)\le 1$ (the
  star) the convexity of $b$ (Lemma~\ref{convex:lem}) gives the result.
\end{proof}

Indeed, the proof of convexity gives an explicit construction for all
$\alpha\in[2,3]$ -- namely the graph with $n$ vertices together with a
root vertex, consisting of $(\alpha-2)n/2$ triangles and $(3-\alpha)n$
edges intersecting in a common root.

\subsection{Mixed Constructions}

For $\a\geq3$ we have the bound $b(\a)\leq\frac{(\a-1)}{\a(\a-2)}$ coming from the random $\a$-regular graph (currently this is just a heuristic bound but it will be made rigorous shortly). Clearly this only holds for integer $\a$. However, we can use the fact that the graph of $b(\a)$ against $\a$ is convex (Lemma \ref{convex:lem}) to combine constructions. For instance, the graph of the convex hull of the points $(2,1)$ and the points $\{(\a,\frac{(\a-1)}{\a(\a-2)}):\a=3,4,\dots\}$ is an upper bound for $b(\a)$, achieved by adding leaves to the root of a random $\a$-regular graph. Unfortunately, this does not give a better bound than Lemma \ref{star:lem}. That is, we cannot improve the bound $b(3)\leq2/3$ by this process.

However, the convex hull of $(2,1)$ and the curve $\{(\a,\frac{(\a-1)}{\a(\a-2)}):\a\in[3,\infty)\}$ does give a better bound. Of course, random regular graphs only exist for integer $\a$. However, if we could achieve something close to the random regular heuristic at non-integer~$\a$ we would expect to be able to give an improved construction at, for instance, $\a=3$. Perhaps suprisingly, this can be done. 

To achieve this we first prove a bound on the average resistance based on the `local' resistance from a vertex out into the graph (recall that this was the basis of the heuristic argument in the Introduction). This is Theorem \ref{t:main-rooted-rrg-lemma} in the next section. The setting for this is a slightly modified model which we will call a random $p$-rooted graph. This result will allow us to prove bounds close to the analogue of the random regular heuristic bound for our constructions with non-integer average degree. Specifically, we will show that there is a graph $G$ with $3<\a<4$ and $B(G)$ (calculated via our `local' resistance result) only a little larger than $\frac{(\a-1)}{\a(\a-2)}$.

It is worth observing that if we were only interested in proving the heuristic bound for random regular graphs then a shorter argument using eignevalue methods could be used. The distribution of eigenvalues of the adjacency matrix of a random $\a$-regular graph was determined by McKay \cite{MCKAY1981203}. An $\a$-regular graph with adjacency matrix $M$ has Laplacian matrix $\alpha I_n-M$ and so McKay's result also gives the distribution of Laplacian eigenvalues from which $A(G)$ can be calculated. It follows that if $G$ is a random $\a$-regular graph on $n$ vertices then with high probability $A(G)=\frac{2(\alpha-1)}{\alpha(\alpha-2)}+o(1)$ (we thank L\'aszl\'o Lov\'asz and Bojan Mohar for pointing out this eigenvalue argument). However, it is important for our application that the bound holds for non-regular graphs for which we need the $p$-rooted model and the full strength of Theorem \ref{t:main-rooted-rrg-lemma}

\subsection{Bounds Using Local Resistance}\label{s:regular}

We start by defining the graph construction we will use. The construction is based on a random rooting similar
to that used in Theorem~\ref{root-unroot:t}. We will prove an upper bound on the average resistance of graphs constructed by this random rooting which shows, for instance, that applying it to any $\a$-regular graph $G$ with suitably large girth gives a graph $G'$ with $A(G')=\frac{2(\a-1)}{\a(\a-2)}+o(1)$.
\begin{defn}
  Suppose that $G$ is a graph and $0<p <1$. Define $\widehat G$ \emph{
    the random $p$-rooted graph formed from $G$} to be the following
  (random) rooted graph.  Form a random set $S$, the \emph{sinks}, by
  including each vertex independently with probability $p$.  Add a root
  vertex $\rho$, and for each vertex $x\in S$ add $d_G(x)-1$ parallel
  edges from $x$ to $\rho$.
\end{defn}

First we prove a lemma about the resistance of this construction when
applied to a small tree. Since we will be considering several
different trees it is convenient to make the following definition.
\begin{defn}
  Suppose that $T$ is a tree and $x\in V(T)$. Then $R(x,T)$,
  \emph{the resistance of the tree from $x$}, is the resistance between 
  $x$ and the leaves of the tree identified to a single vertex.
\end{defn}

For a tree $T$ we would like to bound the expected resistance of $x$
to the root in $\widehat T$. However, there is a positive chance that
there are no edges between the root and $T$ which would give an
infinite resistance to the root. Instead, we show that the the
resistance in $\widehat T$ is `unlikely to be large'.

\begin{lemma}\label{l:rooted}
  Let $0<\eps<1$ be fixed.  Suppose that $T$ is a tree, $x\in V(T)$,
  and the depth of $T$ from $x$ is $\ell$. Let $\widehat T$ be the
  random $\eps$-rooted graph formed from $T$. Then,
  \[
  \Prb\left(R_{x\rho}(\widehat T)>(1+\eps)R(x,T)\right)\le \frac{4(1-\eps)^\ell}{\eps}
  \]
\end{lemma}
\begin{proof}
  Let $D_\ell$ denote all the vertices of $T$ at distance exactly $\ell$ from $x$,
  and let $N=T\setminus D_\ell$.  We consider the paths from $x$
  to~$D_\ell$. We define the $T$-current in a path to be the current
  in the final edge when a current of size 1 flows from $x$
  to~$D_\ell$. This means that the current in any edge is the sum of
  the currents in all the paths through that edge.

  Let $S$ be the set of sinks.  We say a path from $x$ to $D_\ell$ in
  $T$ is \emph{good} if it meets a vertex of $S\cap N$ and \emph{bad}
  otherwise. Let $T_S$ denote the subgraph of $T$ formed by taking the
  union of all good paths. Further, let $\delta$ be the sum of the
  currents flowing in the bad paths in the $T$-flow.

  The proof consists of three steps. First, we show that the
  resistance $R_{x\rho}(\widehat T)\le R(x,T_{S})$. Secondly, we show that
  $R(x,T_{S})\le (1-\delta)^{-2}R(x,T)$. Finally we put these results
  together to prove the bound in the statement of the Lemma.

\subsubsection*{Step 1: $R_{x\rho}(\widehat T)\le R(x,T_{S})$}
Let $N_S$ be the vertices on the good paths of $T$ up to and
including the first vertex in $S$ on the path.

Let $T'$ be the subgraph of $\widehat T$ induced by
$N_S\cup\{\rho\}$. Since $T'$ is a subgraph of $\widehat T$,
$R_{x\rho}(\widehat T)\le R_{x\rho}(T')$.

Now consider the graph $T_S'$ obtained from $T_{S}$ by identifying all
vertices after the first vertex in $S$ (which necessarily includes all
of $T_S\cap D_\ell$) together to one vertex $u$. It is easy to see
that $T'$ is isomorphic to $T_S'$. (This is why we chose to add
$d(x)-1$ edges to each vertex $x$ in $S$).  Since it is formed by
identifying vertices in $T_{S}$ we see that
$ R(x,T_{S})\ge R_{xu}(T_S')= R_{x\rho}(T')$.  Combining this with the
previous inequality completes this step.

\subsubsection*{Step 2: $R(x,T_S)\le (1-\delta)^{-2}R(x,T)$}
We construct a current flow in $T_S$ by restricting the flow in $T$ to
$T_S$: that is we just delete all the flow that exits from a vertex of
$D_\ell\cap (T\setminus T_S)$.

Obviously the current in an edge is no greater than in the original
$T$-flow. The power of the original flow in $T$ is exactly the resistance
$R(x,T)$ and hence the power of the flow in $T_S$ is at most
$R(x,T)$.

The total current of this flow in $T_S$ is only $1-\delta$. However,
if we multiply all currents of this flow by $(1-\delta)^{-1}$ then we
have a unit current from $x$ to $D_\ell$. The power dissipated in each
edge has gone up by a factor of $(1-\delta)^{-2}$ and, therefore, the
power of the new flow is at most $(1-\delta)^{-2}R(x,T)$.  Since the
resistance $R(x,T_S)$ is the minimum power over all unit flows this step
is complete.

\subsubsection*{Step 3}

The probability that a particular path is bad is $(1-\eps)^\ell$. Thus,
$\E(\delta)=(1-\eps)^\ell$.  Thus,
\begin{align*}
\Prb\left(R_{x\rho}(\widehat T)>(1+\eps)R(x,T)\right)
&\le\Prb\left(\delta>1-\frac{1}{\sqrt{1+\eps}}\right)\qquad&\text{by Steps 1 and 2}\\
&\le\Prb\left(\delta>\frac{\eps}{4}\right)&\text{since $\eps<1$}\\
&\le \frac{\E(\delta)}{\eps/4}&\text{Markov's Inequality}\\
&\le \frac{4(1-\eps)^\ell}{\eps}
\end{align*}
as required.
\end{proof}
We are ready to prove the main theorem of this section which bounds the average resistance (rooted or unrooted)
based on the `local' resistance from a point out into the general
graph. Since we can easily calculate this local resistance we can use
it to construct graphs in which we can bound the average resistance.
\begin{theorem}\label{t:main-rooted-rrg-lemma}
  Suppose that $(G_n)_{n=1}^\infty$ is a sequence of graphs with the
  following properties. $G_n$ has $n$ vertices, average degree at
  most~$\a$,  and girth at least~$2\ell+2$ with
  $\ell=\ell(n)$ tending to infinity with~$n$.

  For a vertex $x$, let $T_x$ be the subgraph of $G_n$ formed from all
  vertices at distance at most $\ell$ from $x$ (which is necessarily a
  tree by the girth condition).

  Then, for any $\eps>0$ and all sufficiently large $n$, there are
  rooted graphs $G'_n$ formed from $G_n$ with average degree at most
  $\a+\eps$ such that, for all~$x\in V(G)$, 
  \[
  R_{x\rho}(G'_n)\le (1+\eps)R(x,T_x).
  \]
  In particular, the rooted resistance $B(G'_n)$ satisfies
  \[
  B(G'_n)\le \frac{1+\eps}{|G_n|}\sum_{x\in G_n}R(x,T_x).
  \]
\end{theorem}
\begin{proof}
  We work with each $G_n$ in turn so, for notational simplicity, fix
  $G=G_n$. Let $p=\eps/8\a$ and $\widehat G$ be the random
  $p$-rooted graph formed from $G$ and let $S$ be the sinks of this
  graph.  The expected number of edges we add is
  $p\sum_xd(x)=p\alpha n=\eps n/8$. Hence, with probability at least
  one half we do not add more than $\eps n/4$ edges.
  
  Now, for each vertex $x$, the graph $\widehat T_x$ is a subgraph of
  $\widehat G$ so $R_{x\rho}(\widehat G)\le R_{x\rho}(\widehat T_x)$.  By
  Lemma~\ref{l:rooted}, we have
  \[
  \Prb\left(R_{x\rho}(\widehat G)>(1+\eps)R(x,T_x)\right)\le \frac{4(1-\eps)^\ell}{\eps}=o(1).
  \]

  Let $S'=\{x\in G:R_{x\rho}(\widehat G)>(1+\eps)R(x,T_x)\}$ be the
  set of vertices that do not satisfy this and let
  $D=\sum_{x\in S'}d_G(x)$. Then
  \[\E(D)=\sum_{x\in V}\Prb(x\in S')d_G(x)\le
    \frac{4(1-\eps)^\ell}{\eps}\sum_{x\in V}d_G(x)=o(n)\] so, by
  Markov's inequality, $\Prb(D>\eps n/4)=o(1)$.

  Combining these two bounds we see that with positive probability both the
  number of edges from $S$ to the root is at most $\eps n/4$ and
  $D\le \eps n/4$. Fix $\widehat G$ to be a graph satisfying both of
  these properties.  Form $G'$ from $\widehat G$ by adding $d_G(x)$
  edges between each vertex $x$ of $S'$ and the root.

  We claim all vertices $x$ in $G'$ satisfy
  \[
  R_{x\rho}(G')\le(1+\eps)R(x,T_x).
  \]
  Indeed vertices not in $S'$ trivially satisfy this condition and,
  since we joined each  vertex $x$ in $S'$ to the root
  with $d(x)$ edges, $R_{x\rho}(G')\le 1/d(x)\le R(x,T_x)$ for all
  vertices in $S'$.

  Finally we just need to check the average degree of $G'$. There are
  at most $\eps n/4$ edges from $S$ to the root and we added $D\le
  \eps n/4$ edges from $S'$ to the root. Hence we have added at most
  $\eps n/2$ edges in total which increases the average degree by most
  $\eps$: i.e. the average degree of $G'$ is at most $\a+\eps$ as
  claimed.
\end{proof}

To illustrate Theorem~\ref{t:main-rooted-rrg-lemma}, we prove the bound given by the heuristic
for random regular graphs. Essentially we just need to observe that
for any integer $\a\ge 3$ there exist $\a$-regular graphs with girth
tending to infinity.  Indeed a random $\a$-regular graph has girth at
least $\ell$ with positive probability; see Bollob\'{a}s~\cite{MR1864966} for details. Alternatively, for some
explicit constructions with much stronger bounds on the girth see Lubotzky, Phillips and Sarnak~\cite{MR963118}. 

\begin{corollary}\label{rrr:cor}
For $\a\in\{3,4,\dots\}$ we have that
$b(\a)\leq\frac{\a-1}{\a(\a-2)}$ and, thus, $a(\a)\leq \frac{2(\a-1)}{\a(\a-2)}$.
\end{corollary}
\begin{proof}
  Fix $\eps>0$ and let $G_n$ be an $\a$-regular graph of girth tending
  to infinity with $n$, and form the graph~$G'$ as in
  Theorem~\ref{t:main-rooted-rrg-lemma}. In a regular graph all the
  trees $T_x$ are the same and have resistance less than the
  corresponding infinite tree: i.e., $R(T_x)\le
  \frac{\a-1}{\a(\a-2)}$. Thus $G'$ has average degree at most
  $\a+\eps$ and average resistance at most $
  \frac{\a-1}{\a(\a-2)}+o(1)$.

  Hence $b(\a+\eps)\leq\frac{\a-1}{\a(\a-2)}$ for all $\eps>0$ and the
  continuity of $b$ (Lemma~\ref{convex:lem}) shows $b(\a)\le
  \frac{\a-1}{\a(\a-2)}$.  This, together with
  Theorem~\ref{root=unroot:t}, implies the bound for~$a$.
\end{proof}

As we noted earlier, this bound can also be proved more directly by eigenvalue arguments. In the next section we will prove Theorem~\ref{t:alpha=3} using a similar application of Theorem~\ref{t:main-rooted-rrg-lemma} to non-regular graphs where
eigenvalue methods are not enough.

\subsection{Proof of Theorem~\ref{t:alpha=3}}

Recall that our aim is to construct a graph $G$ with $3<\a<4$ and $B(G)$ only a little larger than $\frac{(\a-1)}{\a(\a-2)}$. 

For non-integer values of $\a$ we certainly have (using Lemma
\ref{convex:lem} again) that the convex hull of the points
$\{(\a,\frac{(\a-1)}{\a(\a-2)}):\a=3,4,\dots \}$ is an upper bound for
$r(\a)$. This corresponds to taking a union of a rooted random
$\lfloor\a\rfloor$-regular graph and a rooted random
$\lceil\a\rceil$-regular graph which intersect only in the root
vertex. However, as we shall see, if we take a high girth graph of
average degree $\a$ in which all vertices have degree
$\lfloor\a\rfloor$ or $\lceil\a\rceil$, then the average rooted
resistance of a rooting of this graph is slightly lower. Hence, it is
plausible that by taking such a rooting and adding some leaves (i.e.,
taking the convex hull in~$b$) we would get an improved bound for
$3$ and, indeed, this is the case.

By taking care that the vertices of different degrees are suitably
distributed in the graph we can do even better. A simple example of
such a construction is a high girth bipartite graph with bipartition into a
part containing $\frac{3n}{7}$ vertices of degree 4 and a part
containing $\frac{4n}{7}$ vertices of degree 3. This graph has average
degree $\frac{24}{7}$ and has 2 kinds of trees $T_x$ depending on
whether $x$ has degree 3 or degree 4. Calculating the resistance of
these trees and applying Theorem~\ref{t:main-rooted-rrg-lemma} yields
a bound of $B(G)=\frac{209}{420}+o(1)\approx0.4976$. We omit the
details since the next example is better for our constructions and
involves very similar calculations

Let $G_0$ be a high girth 4-regular bipartite graph. It is clear that
such graphs exist since, for example, a random 4-regular bipartite
graph has a positive chance of having girth greater than any fixed
size. This is very similar to the results of random regular graphs
discussed above: see~\cite{MR1725006} for full details. 

Then form $G$ by replacing each vertex in one partition by two
vertices joined by an edge, with the four neighbours being shared two
to each new vertex. Then $G$ has average degree $\frac{10}{3}$. We
calculate the `resistance to infinity' in the related tree: that is in
the tree obtained from a 4-regular tree by replacing all vertices at
even distance from some designated vertex with two vertices of degree
3. Let $x$ be the resistance to infinity in this tree from a degree 3
vertex along one of the edges to a degree 4 vertex, and let $y$ be the
resistance to infinity from a degree 4 vertex along any one of its
edges. Then, by a simple application of the series and parallel laws, we have
\[
x=1+\frac y3 \qquad\text{and}\qquad y=1+\frac{1}{\frac 1x+\frac1{1+x/2}}.
\]
Solving this gives $x=\frac{1+\sqrt{5}}2$ and $y=\frac{-3+3\sqrt 5}2$,
which corresponds to a resistance to infinity from a degree 3 vertex
of $1/(2/x+1/(1+x/2))=\frac{\sqrt{5}}4$ and from a degree 4 vertex of
$y/4=\frac{-3+3\sqrt 5}8$. Thus, the average resistance to infinity is
$\frac{-3+7\sqrt5}{24}\approx 0.527186$.

By applying Theorem~\ref{t:main-rooted-rrg-lemma}, we can construct a
rooted graph with average degree $\frac{10}{3}+o(1)$ and average rooted
resistance $0.5271865>\frac{-3+7\sqrt5}{24}$. We have proved the following theorem which
includes Theorem~\ref{t:alpha=3} as a special case.
\begin{theorem}\label{mixed:thm}
For $\a\in[2,10/3]$ we have that $b(\a)\leq 1-0.3546(\a-2)$. In particular
  $b(3)\leq 0.6454$ and $a(3)\le 1.2908$.
\end{theorem}
\begin{proof}
  This follows instantly from convexity, $b(2)=1$ and the above
  construction showing that $b(10/3)\le 0.5271865$.
\end{proof}

\section{Lower Bounds}\label{s:lower}

In the introduction we proved that 
\[
A(G)\ge \frac{2}{(\frac 1n\sum_xd_x)}+O(m/n^2)
\]
which showed that $a(\alpha)\ge 2/\alpha$ or, equivalently, that $b(\a)\ge
1/\a$. The rough idea was to bound the resistance from a vertex to the
rest of the graph by the resistance to its neighbourhood. We improve
this bound by considering the resistance from a vertex to its two-step
neighbourhood. It is more convenient to work in the rooted case.

\begin{theorem}\label{t:2-step-lower}
For all $\a\ge 2$ we have  $b(\alpha)\ge 1/(\alpha-1)$.
\end{theorem}
\begin{proof}Let $G$ be an optimal graph and let $G_x$ be the graph
  formed from $G$ by identifying all vertices not in $x\cup \Gamma(x)$
  (i.e., all vertices at distance at least two from $x$) to the root
  vertex. This decreases the resistance from $x$ to the root. We call
  the resistance from $x$ to the root in $G_x$ the \emph{two-step
    resistance} of $x$ and similarly for the \emph{two-step
    conductance}.

  Let $d_x$ be the degree of a vertex and let $e_x$ be the number of
  those edges which join it to the root. Thus
  \[
  \sum_{x\not=\rho}(d_x+e_x)=\alpha n
  \]
  Let $\Gamma'$ denote the neighbourhood of $x$ viewed as a
  multiset and excluding the root vertex. Thus
  $|\Gamma'(x)|=d_x-e_x$. Note, that for any (non-root) vertex $x$ and
  vertex $y\in\Gamma'(x)$ we have $d_y\geq2$: indeed, otherwise $y$
  would be a leaf vertex not joined to the root contradicting the
  assumed optimality of $G$.

  Having set up the notation observe that the two-step conductance
  from a vertex $x$ is at most
  \[
  \sum_{y\in\Gamma'(x)}\left(1+\frac1{d_y-1}\right)^{-1}+e_x.
  \]
  To complete the proof we sum this bound over all vertices and
  simplify.  To reduce the notation all vertex sums are over all vertices
  except the root.  The sum of the two-step conductances is at
  most
\begin{align*}
  \sum_x\left(\sum_{y\in\Gamma'(x)}\left(1+\frac1{d_y-1}\right)^{-1}+e_x\right)
  &=\sum_x\sum_{y\in\Gamma'(x)}\frac{d_y-1}{d_y}+\sum_x e_x\nonumber\\
  &=\sum_x\sum_{y\in\Gamma'(x)}1-\sum_x\sum_{y\in\Gamma'(x)}\frac{1}{d_y}+\sum_x e_x\nonumber\\
  &=\sum_xd_x-\sum_y\sum_{x\in\Gamma'(y)}\frac{1}{d_y}\nonumber\\
  &=\sum_xd_x-\sum_y \frac{d_y-e_y}{d_y}\nonumber\\
  &=\sum_x(d_x+e_x)-\sum_y \frac{d_y-e_y}{d_y}-\sum_ye_y\nonumber\\
  &=\sum_x(d_x+e_x)-\sum_y1 + \sum_y \frac{e_y}{d_y}-\sum_ye_y\label{e:two-stage-line}\\
  &\le\alpha n-n=(\alpha-1)n.\nonumber
\end{align*}
Hence the average two-step conductance is at most $\alpha-1$ and,
thus, the average two-step resistance is at least $1/(\a-1)$.
\end{proof}
We might expect that, with a bit more ingenuity and algebraic
manipulation, we could extend the above to 3 and more steps. This does
not appear to be easy and, indeed, the natural analogue of
Theorem~\ref{t:2-step-lower} for five or more steps would contradict
Theorem~\ref{mixed:thm}, so cannot possibly work.

We see that the lower bound for $\alpha=2$ now matches the star
construction. Recall from Lemma~\ref{convex:lem} that we needed this
result to prove the continuity of $b$ (and thus $a$) at 2.

\section{The Structure of Optimal Graphs}\label{s:optimal-graphs}

In this section we prove some results describing what optimal graphs
look like. One particular aim is to show that, as $\a$ increases, the
graph changes from `star-like' to `regular-like'.

First we show that for large $\a$ no optimal graph, in either the rooted or unrooted case,
has a significant number of leaves. This is a simple application of the
convexity of $b$.

\begin{lemma}\label{l:no-leaves}
  Suppose that $\a\ge 3.83$ and that $G$ is any graph with average degree
  $\a$ and $\gamma n$ leaves. Then $A(G)-a(\a)\ge\gamma/100+o(1)$. In
  particular the optimal graph has $o(n)$ leaves.
\end{lemma}
\begin{proof}
  We start by proving a similar result for the rooted case.
  Trivially, $b(2)=1$, and, by Theorem~\ref{mixed:thm}, $b(10/3)\le
  0.527186$. Also, by Theorem~\ref{t:2-step-lower},
  $b(\a)>1/(\a-1)$. Let $l=l(x)$ be the function giving the line
  through the points $(10/3,0.527186)$ and $(\a,b(\a))$, and $l_0$ the
  line through $(10/3,0.527186)$ and $(3.83,1/2.83)$.  Since $b$
  is convex and $\a\ge 3.83>10/3$ we see that for all $x\ge \a$ we have
  $b(x)\ge l(x)$. Note that, $l_0(2)\le 0.995$ so, again by
  convexity, $l(2)\le 0.995$.

  Split $G$ into two subgraphs disjoint except for the root: the
  leaves $L$, and the rest of the graph $H$. The average degree
  of $G$ satisfies
  \[\a(G)=\gamma\alpha(L)+(1-\gamma)\alpha(H)=2\gamma
  +(1-\gamma)\alpha(H),\] and, in particular, $\a(H)\ge \a(G)$.

  Turning to resistances, we have
  \begin{align*}
    B(G)&=\gamma B(L)+(1-\gamma) B(H)\\
    &\ge \gamma+(1-\gamma)b(\alpha(H))\\
    &\ge \gamma l(2)+(1-\gamma)l(\alpha(H))+(1-0.995)\gamma\\
    &=l(\a(G))+\gamma/200\\
    &=b(\a(G))+\gamma/200.
  \end{align*}
  
  Finally, we lift this result back to the unrooted case. Let $G'$ be
  a rooting of the graph $G$ with $2 B(G')\le A(G)+o(1)$ and
  $\a(G')\le \a(G)+o(1)$. Since the rooting process only added $o(n)$
  edges to $G$ we see that $G'$ has $\gamma n-o(n)$ leaves.

  Thus 
  \begin{align*}
  A(G)&\ge 2 B(G')+o(1)\\
  &\ge 2b(\alpha(G'))+\gamma/100+o(1)\\
  &\ge 2b(\alpha(G))+\gamma/100+o(1)\\
  &=a(\alpha(G))+\gamma/100+o(1)
\end{align*}
as required.
\end{proof}
We remark that, if we just used the simple bound $b(3)\le 2/3$ given
by a star of triangles, then we would prove a similar result for the
slightly weaker case when $\a\ge 4$.

\medskip

Next we show that, if $\a$ is not much greater than 2, then
the optimal graph must have leaves. 

\begin{theorem}\label{leaves:thm-better}
  Let $G$ be a graph on $n$ vertices with average degree
  $\alpha=2\frac14-\eps$, for some $\eps>0$, and $\gamma n$ leaves.
  Then $A(G)-a(\a)\ge \frac{2}{25}(\eps-\gamma/4)+o(1)$.  In
  particular, the optimal such graph must contain at least $4\eps
  n-o(n)$ leaves.
\end{theorem}
\begin{remark}
  The bound of $\a<2\frac14$ can be improved significantly but the
  important point for our purposes is that the bound is strictly
  greater than 2.  We discuss the limits of the basic technique after
  the proof.
\end{remark}

Roughly, our aim is to show that there is a vertex $x$ with resistance
to the root at least $1$; then we can contract an edge incident to $x$
reducing the total resistance by 1 and add a leaf.

We will be considering a number of modified graphs. Let $G$ be the
original graph and suppose that $x$ is a degree-2 vertex with
neighbours $y_1$ and $y_2$. Let $G\setminus x$ denote the graph with
vertex $x$ deleted; $G/x$ denote the graph with the edge $y_1x$
contracted (i.e., $G\setminus x$ with the edge $y_1y_2$
added). Finally, let $G^+/x$ denote $G/x$ with an additional leaf
added to the root so, in particular, $G^+/x$ has the same number of
vertices and edges as $G$.

\begin{lemma}\label{l:weight-resistance}
  Let $G$ be a rooted graph with no leaves and average degree less
  than $9/4$. Then  there exists a vertex $x$ with
  $\rtot(G)-\rtot(G^+/x)\ge \frac1{100}$.
\end{lemma}
\begin{proof}
  We claim that there is a degree-2 vertex which has both neighbours
  also degree-2. Indeed, suppose not.  Let $S$ be the set of vertices
  of degree $2$.  Since $G$ has no leaves, the sum of the degrees
  is at least $2|S|+3|V\setminus S|\le \frac94n$, so $|S|\ge
  3n/4$. Let $t$ be the number of edges from $S$ to $V\setminus
  S$. Since no vertex in $S$ has both neighbours in $S$ we have $t\ge
  |S|$. Thus, by considering the degrees in $G[S]$ we see that
  $2E(S)=2|S|-t$. The total number of edges in $G$ is at least
  \[
  t+|E(S)|=t/2+|S|\ge \frac32|S|\ge \frac98n
  \]
  which contradicts the average degree being less than $9/4$.

  Let $x$ be such a vertex, and $y_1,y_2$ its two neighbours, and
  $z_1,z_2$ their two neighbours (we do not exclude $z_1=z_2$). We
  immediately have $R_{x\rho}\ge 1$ but we need something slightly
  stronger. If either of $z_1,z_2$ has resistance to the root in
  $G\setminus \{x,y_1,y_2\}$ at least $\frac1{24}$ then
  $R_{x\rho}(G)\ge \frac{98}{97}$. Otherwise we consider the
  contribution to $\rtot$ of $y_1$.  We have $R_{y_1\rho}(G)\ge 3/4$
  and $R_{y_1\rho}(G/x)\le \frac1{24}+\frac23$. In either case we get
  that $\rtot(G)-\rtot(G^+/x)\ge 1/100$ giving
  the result.
\end{proof}

\begin{lemma}\label{l:leaves-rooted}
  Fix $n<m$ with $2m<9n/4$, and let
  $\ell=9n-8m$. Suppose that $G$ is a rooted graph
  with $n$ non-root vertices, $m$ edges, and $k\le \ell$ leaves. Then
  there exists a rooted graph $G'$ also with $n$ vertices and $m$
  edges and $\rtot(G')\le \rtot(G)-\frac1{100}(\ell-k)$.
\end{lemma}
\begin{proof}
  We prove this by induction on $\ell-k$. It is trivial if $\ell-k=0$.
  Thus suppose that $\ell-k\ge 1$ and that the result holds for any
  graph with smaller `$\ell-k$'.

  Let $L$ be the set of vertices of degree one (i.e., the leaves), and
  let $G_0=G[V\setminus L]$. Then $G_0$ has $n-k$ non-root vertices
  and $m-k$ edges. Thus $G_0$ has average degree
  \[
  \frac{2(m-k)}{n-k}<
  \frac{2(m-\ell)}{n-\ell}=\frac{18(m-n)}{8(m-n)}=\frac94.
  \]
  Thus, by Lemma~\ref{l:weight-resistance}, there exists $x'$ with
  $\rtot(G_0)-\rtot(G_0^+/x')\ge \frac1{100}$. Trivially, this implies
  that $\rtot(G)-\rtot(G^+/x')\ge \frac1{100}$.

  Now, $G^+/x'$ has $n$ vertices and $m$ edges, and $k+1$
  leaves. Hence, by induction, there exists $G'$ with $n$ vertices and
  $m$ edges and
  \[
  \rtot(G')\le \rtot(G^+/x')-\frac1{100}(\ell-(k+1)) \le
  \rtot(G)-\frac1{100}(\ell-k),
  \] as claimed.
\end{proof}

\begin{proof}[Proof of Theorem~\ref{leaves:thm-better}]
  Suppose that $G$ is as in the theorem. By
  Theorem~\ref{t:main-rooted-rrg-lemma} we let $G'$ be a rooting of
  $G$ with $B(G')<A(G)/2+o(1)$ and $m'=(1+o(1))m$ edges. Observe
  that $G'$ has at most as many leaves as $G$, so $G'$ has at most
  $\gamma n$ leaves. Thus, by Lemma~\ref{l:leaves-rooted}, there
  exists $G''$ with the same number of vertices and edges as $G'$ and
\[    \rtot(G')-\rtot(G'')\ge \frac1{100}\left(9n-8m'-\gamma n\right)\]
  Therefore
  \begin{align*}
    B(G')-B(G'')&= \frac{\rtot(G')-\rtot(G'')}{n}\\
    &\ge \frac{9n-8m'-\gamma n}{100n}\\
    &=    \frac{9-4\a-\gamma }{100}+o(1)\\
    &\ge    \frac{4\eps-\gamma }{100}+o(1).
  \end{align*}
  Finally, $B(G'')\ge a(2m'/n)/2=a(\alpha)/2+o(1)$. Putting this all together we have
  \begin{align*}
    A(G)&\ge 2B(G')+o(1)\\
    &\ge 2B(G'')+\frac{4\eps-\gamma }{50}+o(1)\\
    &\ge a(\alpha)+\frac{2}{25}(\eps-\gamma/4)+o(1).\qedhere
  \end{align*}
\end{proof}

As mentioned above, with substantially more effort, the bound could be
improved significantly: to a little beyond $2\frac12$. However, the
key technique we used was to find a vertex such that contracting one
of its edges and adding a leaf reduces the total resistance.  A star
of triangles shows that there may be no such vertex when $\a\ge
3$. Thus improving the bound beyond this would require new ideas.

We would like to say that there is phase transition where the graph
changes from star-like to regular-like. We have seen that when
$\a>3.83$ the optimal graph cannot contain a positive proportion of
leaves, and when $\a<9/4$ the optimal graph must contain a positive
proportion of leaves. It seems likely that there is a threshold when
the optimal graph changes; i.e., that there is a phase transition. But
we are not able to show this and there could be a region where there
are optimal graphs with, and without, a positive proportion of leaves.

\section{Queen-Bee Model}\label{s:queen-bee}

A naive approach to the rooted problem would be to insist that we join
the root to every other vertex. This uses $n$ edges (since we have
defined the rooted model to have $n$ non-root vertices). As well as
being a natural subclass of rooted graphs, it also has a natural
interpretation in the statistical application where it corresponds to
the situation that every treatment is being compared with one control
treatment.  Following Bailey and Cameron~\cite{MR3026766}, we call
such graphs \emph{Queen-Bee networks}. Note we do allow multiple edges
between the root and non-root vertices; we insist only that there is
at least one edge between the root and each non-root vertex.

In this section we show that the optimal Queen-Bee network for
$\a\in[2,3]$ is the star of triangles and leaves with the correct
number of edges, but that a Queen-Bee network is not optimal (in the
space of all rooted graphs) for any $\a>2$.

We prove our bound by moving into the space of networks where each
edge, including the edges from the root vertex, can have any
non-negative real conductance and, as in the introduction, the total
conductance is bounded by $m$. We say a
configuration is \emph{legal} if the non-root vertices can be partitioned into
`components' satisfying
\begin{enumerate}
\item edges between distinct components have zero conductance
\item each component of size $s$ has a total conductance internally an
integer at least $s-1$, and total conductance to the root an integer
at least $s$.
\end{enumerate}
\begin{lemma}\label{l:qb-component}
  Suppose that $C$ is a component of a legal configuration on $s$ vertices
  with sum of the conductances of the edges $m=2s-1+t$. Then the total
  effective resistance to the root is at least
  $\frac{s+2}3-\frac{2t}{3}=s-\frac{2(m-s)}{3}$. Moreover this
  bound is only obtained if $m=2s-1$.
\end{lemma}
\begin{proof}
  If $t=0$ (the minimum allowed in a legal configuration) then by the
  usual averaging argument the conductances on all edges inside a
  component are equal, and the conductances to the root from a
  component are all equal. Thus, by the definition of legality we know
  that each edge to the root has conductance~1, and each edge in the
  component has conductance $2/s$. Easy calculation shows that the
  resistance to the root is $\frac{s+2}{3}$ and the bound is tight in
  this case.

  Now suppose that $t>0$.  If $s=1$ then the resistance to the root is
  $\frac{1}{t+1}$ which is at least $1-\frac{2t}{3}$ for all positive
  integers. Hence the only remaining case is $t>0$ and $s>1$. 

  We do not know (we could calculate but it is not informative)
  whether to put the extra conductance in the component or to the
  root. However, we just do both: we show that the bound still holds
  if the total conductance in the component is $s-1+t$ and the total
  conductance to the root is $s+t$. Indeed, if we take the example
  above on $s$ vertices for $t=0$ and multiply each conductance by
  $1+t/(s-1)$ then we get a network with conductances at least this
  large and the resistance to the root goes down by
  \begin{align*}
    \frac{s+2}{3}\left(1-\frac{1}{1+\frac{t}{s-1}}\right) 
    &=\frac{s+2}{3}\left(1-\frac{s-1}{s-1+t}\right)
    &=\frac{s+2}{3}\frac{t}{s-1+t}
    \le \frac{2}{3}t
  \end{align*}
  with equality only if $s=2$ and $t=1$. It is easy to verify that the
  bound in this case is strict too.
\end{proof}
\begin{theorem}\label{t:qb-2-3}
  Suppose that $G$ is a Queen-Bee network with $n$ vertices and average degree
  $\a$. Then the average resistance to the root $B(G)$ is at least
  $\frac{5-\a}{3}$. Moreover, for $\a\in [2,3]$ this is attainable.
\end{theorem}
\begin{proof}
  Let $m=\a n/2$.  Split the graph (without the root) into
  components. Each component is necessarily a legal component in the
  above definition. We write $n_i$ for the sizes, and $m_i$ for the
  weights, of the components.  By Lemma~\ref{l:qb-component} the total
  resistance to the root is at least
  \[
  \sum_i \left(n_i-\frac{2(m_i-n_i)}{3}\right)=n-\frac{2m-2n}{3}=\frac{n (5-\a)}{3}
  \]
  and, thus, the average resistance to the root is at least $\frac{5-\a}{3}$.

  Finally, to see that this is attainable when $\a\in [2,3]$
  consider a star of leaves and triangles with the appropriate number
  of edges. The bound in Lemma~\ref{l:qb-component} is attained for both
  leaves and triangles (components with $t=0$ and either $r=1$ or
  $r=2$) and hence the above bound is attained.
\end{proof}
Next we improve this bound for large $\a$.
\begin{lemma}\label{l:qb-all-a}
  Suppose that $G$ is a Queen-Bee network with average degree $\a$. Then
  \[
  B(G)\ge
  \begin{cases}
    \frac{5-\a}3\qquad&\text{if $\a\le 3\frac12$}\\
    \frac{1}{\a-3/2}\qquad&\text{if $\a\ge 3\frac12$}
  \end{cases}
  \]
\end{lemma}
\begin{proof}
  Let $f=f(\a)$ be the claimed lower bound (i.e., the right hand side
  above) and note that $f$ is convex.  

  We start by showing that, if $G$ has no leaves, then
  $B(G)\ge 1/(\a-3/2)$.  This is actually an easy consequence of the
  proof of the two-step lower bound
  (Theorem~\ref{t:2-step-lower}). The penultimate line of that proof
  showed that the sum of the two-step conductances is at most
  \begin{equation}
  \sum_x(d_x+e_x)-\sum_y1 + \sum_y \frac{e_y}{d_y}-\sum_ye_y\label{e:two-stage-line2}    
  \end{equation}
  In that proof we just used that the sum of the final pair of terms,
  \[
  \sum_y \frac{e_y}{d_y}-  \sum_ye_y
  \]
  is negative. However, for a Queen-Bee network we know that
  $e_y\ge 1$ for all $y$.  Since $G$ has no leaves, $d_y\ge 2$ and we
  see that $e_y-e_y/d_y\ge 1/2$ for all vertices $y$.  Substituting
  this improved bound into equation~(\ref{e:two-stage-line2}) we get
  \[
  \sum_x(d_x+e_x)-\sum_y1 + \sum_y \frac{e_y}{d_y}-\sum_ye_y
  \le \a n-n-n/2    = (\a -3/2)n
  \]
  and thus the average resistance to the root in $G$ is at least
  $\frac{1}{\a-3/2}$.

  By Theorem~\ref{t:qb-2-3} we see that $B(G)\ge \frac{5-\a}3$ so
  combining these we have $B(G)\ge f(\a)$.

  \smallskip So far we have only considered the case when $G$ has no
  leaves, and we turn now to the general case.  First observe that we
  may assume that all leaves are joined to the root, since we can move
  them to the root without increasing the resistance.

  Given any such Queen-Bee network $G$, we write it as the union
  of a Queen-Bee network $G_1$ and a star $G_2$ (i.e., the union of
  some leaves) sharing a common root but otherwise disjoint. Let
  $\alpha_1$ be the average degree of $G_1$ and $\alpha_2=2$ the
  average degree of $G_2$. Since, by the above, $B(G_1)\ge
  f(\alpha_1)$, and, trivially, $B(G_2)=1=f(\alpha_2)$, the convexity
  of $f$ implies that $B(G)\ge f(\alpha)$.
\end{proof}

These bounds allow us to show that, with the exception of $\a=2$
(essentially the star), no Queen-Bee network is optimal.

\begin{theorem}\label{t:qb-never-best}
No Queen-Bee network with average degree strictly greater than 2 is optimal.
\end{theorem}
\begin{proof}
  We just need to show that our (non-Queen-Bee) constructions give
  average resistance strictly less than the lower bound proved in
  Lemma~\ref{l:qb-all-a}.

  We have given constructions of average degree $\a$ and resistance at
  most $r$ for any point in the convex hull of the star ($\a=2$,
  $r=1$), the random $10/3$ regular construction used for
  Theorem~\ref{mixed:thm} ($\a=\tfrac{10}{3}$ $r=0.528$), and the
  regular construction (Corollary~\ref{rrr:cor}) for $\a\ge 4$ with
  $\a\in \N$ (which gives $r=\tfrac{\a}{\a(\a-2)}$). An easy but
  lengthy calculation (given in the appendix) shows that these
  constructions have average resistance less than $1/(\a-3/2)$ for all
  $\a\ge 3\frac12$ and less than $(5-\a)/3$ for $2<\a\le 3\frac12$.
\end{proof}

\section{Open Problems}\label{s:open-questions} Our main open
question is to find the function $a$. To state our conjecture we want
to define a function $f$ that is the maximal convex function which is
less than the value of $b$ given by the star, and less than the random
regular heuristic for $b$. More precisely, let
$f\colon[2,\infty)\to\R$ be the maximal convex function such that
$f(2)\le1$ and $f(x)\le \frac{x-1}{x(x-2)}$ for all
$x\in(2,\infty)$. (Since the supremum of a family of convex functions
is also convex it is clear that this function exists.)

\begin{conj}
  Let $f$ be defined as above.  Then $a(x)\ge 2f(x)$ for all $x$
  (equivalently $b(x)\ge f(x)$).
\end{conj}
\noindent%
In particular, since we have the corresponding upper bound for integer
$\alpha$ at least four, this would imply that
$a(\alpha)=\frac{\a-1}{\a(\a-2)}$ for all such $\alpha$.

Our next two conjectures concern the existence of a phase transition
between star-like and regular-like behaviour as discussed in
Section~\ref{s:optimal-graphs}. 
\begin{conj}
  There is a threshold $\a_0$ on the average degree below which all
  optimal graphs have a positive proportion of leaves, and
  above which all optimal graphs have $o(n)$ leaves.
\end{conj}
\noindent%
In fact we would expect this also to be true for graphs which are not
optimal but are sufficiently close to optimal.  

Of course the existence of leaves is not the only way a graph could
fail to be `regular-like' -- it could have vertices of high degree.
\begin{conj}
  There is a threshold $\a_1$ on the average degree below which the
  optimal graph has a vertex with degree $\Omega(n)$, and above which
  all vertices have degree $o(n)$.
\end{conj}
If these conjectures are true then, obviously, $\a_1\ge \a_0$, but
perhaps they are, in fact, equal. This would imply that removing the
leaves from any optimal the graph would yield a graph with maximum
degree $o(n)$ (since removing the leaves must yield a graph of average
degree greater than $\alpha_0$).

Theorem~\ref{t:qb-never-best} showed that it was not optimal, for any
$\a>2$, for a vertex to have degree $n$. The same argument shows that
if the root of $G$ has $\eta n$ non-leaf neighbours then $B(G)\ge
\frac{1}{\a-1-\eta/2}$. Comparison with our upper bound of
$\frac{\a-1}{\a(\a-2)}$ then shows that, for $\alpha\in \N$, $\eta$ is
at most $\frac{2}{\a-1}$. Since Lemma~\ref{l:no-leaves} shows that for
$\a\ge 4$ the optimal graph has $o(n)$ leaves, we see that the root
has degree at most $2n/(\a-1)$. Since we can short any other vertex to
the root (removing one vertex) this shows that the maximum degree is
at most $2n/(\a-1)$ and, in particular, that the maximum degree (as a
proportion of all vertices) tends to zero as $\a\to \infty$.

We also have an open question, motivated by the statistical
applications which is, perhaps, less natural from a pure mathematical
perspective. It describes the situation when treatments are compared
in blocks of size $r$ for some $r>2$; in comparison the model we have
been considering is the case when the blocks have size 2.
\begin{question}
  Fix $n$, $m$ and $r$. Given any collection $E_1,E_2,\dots, E_m$ of
  the $r$-sets of $[n]$ we can form a graph $G$ on $[n]$ with
  edge (multi)-set $\bigcup_{i=1}^m E_i^{(2)}$ (i.e., the union of the
  cliques on each $E_i$). Which such collection minimises $A(G)$?
\end{question}

Finally, the corresponding question for maximum rather than average resistance is also very natural:
\begin{question}
  Which graphs $G$ with $|V(G)|=n$ and $e(G)\le \frac{\alpha n}{2}$ minimise
    $\max_{x,y\in V(G)}R_{xy}$.
\end{question}
The maximum seems to behave very differently from the average. Indeed,
in the rooted version of the problem, the maximum is 1 for all
$\alpha\in [2,9/4]$ (since the graph either has a leaf, or a vertex of
degree 2 with both neighbours having degree 2, as in the proof of
Lemma~\ref{l:weight-resistance}) whereas the star of triangles shows
the maximum at $\alpha=3$ is $2/3$. In particular, unlike the average,
the maximum is not a convex function.

\bibliography{mybib}{}
\bibliographystyle{abbrv}

\section*{Appendix A: Basic Theory of Electrical Networks}\label{s:en} 
In this appendix we recall the definitions from electrical networks
and the basic facts about such networks that we use in the main paper.
For a general reference see Snell and Doyle~\cite{MR920811} or
Chapters~2 and~9 of Bollob\'as~\cite{MR1633290}.

We start by recalling the definition of the effective resistance
between two vertices $x$ and $y$ in a network $G$. The definition
relies on Kirchhoff's two laws~\cite{kirchhoff} which we recall now. A
flow from $x$ to $y$ consists of an assignment of currents $I_{uw}$ to
each oriented edge $uw$ such that
\begin{itemize}
  \item \textbf{Kirchhoff's Current Law} 
  $\sum_{w\in \Gamma(u)}I_{uw}=0$ for all $u$ except $x$ and~$y$.
\end{itemize}
and it is a \emph{current flow} if it also satisfies
\begin{itemize}
  \item \textbf{Kirchhoff's Potential Law} The sum of the currents around any
  oriented cycle in $G$ is zero.
\end{itemize}

The Potential Law is equivalent to saying that we can assign
potentials (voltages) $V_w$ to each vertex $w$ such that:
\begin{itemize}
  \item \textbf{Ohm's Law} for any oriented edge $uw$ we have $I_{uw}=V_u-V_w$.
\end{itemize}
The definitions above are all for the case we primarily consider in
which all edges have unit resistance.  On a small number of occasions
we will consider graphs where edges have other resistances;
Kirchhoff's current law remain unchanged, Kirchhoff's potential law
becomes that the sum of the products of the resistance and current in
the edges round a cycle is zero, and Ohm's Law takes its more familiar
form of $I_{uw}r_{uw}=V_u-V_w$ where $r_{uw}$ denotes the resistance
of the edge $uw$.

Standard results show that, for any connected network, there is a
unique current flow from $x$ to $y$ in which unit current enters the
network at $x$ and leaves at $y$, and that the potentials are uniquely
defined by this flow up to an additive constant.  The \emph{effective
  resistance} between $x$ and $y$ is $V_x-V_y$ for this flow. The
\emph{effective conductance} is the reciprocal of the effective
resistance. We write these as $R_{xy}(G)$ and $C_{xy}(G)$
respectively; where there is no ambiguity we abbreviate these to
$R_{xy}$ and $C_{xy}$.

\begin{lemma}
  Suppose that $G$ is a connected network and $x,y\in V(G)$ and that
  $(V_x)_{x\in V(G)}$ are the potentials associated with a current flow
    from $x$ to $y$. Then all potentials lie between $V_x$ and $V_y$.
\end{lemma}
\begin{proof}
  Consider the set of vertices of maximum potential. One of them, $w$
  say, must be joined to a vertex of lower potential (or all
  potentials are equal and we are done). By the Potential Law $w$ only
  has current leaving in the network (and does have some current leaving):
  i.e., the current law is not satisfied so $w$ must be one of $x$ and
  $y$. Since the same holds for the minimum potential, the result
  follows.
\end{proof}

On some occasions it will be convenient to allow more than one vertex
where current enters or leaves the networks. More precisely, we may
consider flows where the Potential Law is satisfied everywhere and the
Current Law is satisfied except at some set $S$ of vertices. We call
$S$ the \emph{source-sinks} (vertices where current enters the network are
sources, and the vertices where current leaves are sinks). The first
result using this is that we can add or \emph{superpose} current flows.

\begin{theorem}
  Suppose that $G$ is a network and that $I$ and $J$ are flows with
  source-sinks $S$ and $T$ respectively. Then $I+J$ is a current flow
  with source-sinks a subset of $S\cup T$.
\end{theorem}
This follows immediately from the fact that both the Current Law and
the Potential Law are linear.  This also shows that the current flow
is unique (the difference of two flows would have no source-sinks so
must be the zero flow).

An easy consequence of this is the triangle inequality for
resistances.
\begin{lemma}\label{l:triangle-inequality}
  Suppose that $G$ is a network and that $x,y,z$ are vertices in
  $G$. Then $R_{xz}\le R_{xy}+R_{yz}$.
\end{lemma}
\begin{proof}
  Take a current flow $I$ of size 1 with source $x$ and sink $y$, and
  a current flow $J$ of size 1 with source $y$ and sink $z$, and let
  $V$ and $W$ be the vectors of potentials associated with the flows
  $I$ and $J$ respectively.   Then $I+J$ is a current flow of size 1
  with source $x$ and sink $z$, with associated potentials $V+W$. Thus
  \[
  R_{xz}=V_x+W_x-V_z-W_z\le V_x-V_y+W_y-W_z=R_{xy}+R_{yz}.\qedhere
  \]
\end{proof}

There are two very natural ways to combine two networks. Suppose that $G$
and $G'$ are networks on disjoint vertex sets with $x,y\in V(G)$ and
$z,w\in V(G')$. Then we can combine them `one after the other' by
considering $G\cup G'$ with $y$ and $z$ identified. We call this
joining in \emph{series}. If $H$ is the resulting network then
$R_{xw}(H)=R_{xy}(G)+R_{zw}(G')$.

The second possibility is to take $G\cup G'$ and identify $x$ and $z$
and identify $y$ and $w$. We call this joining in
\emph{parallel}. Writing $H$ for the resulting network this time the
conductances add: $C_{xy}(H)=C_{xy}(G)+C_{zw}(G')$. We remark that
this could be rephrased in a rather more complicated formula for
resistances, but that we find calculations easier in the form we have
stated.

There are several different ways of describing the unique current
flow. It is frequently the case that a result is easy to prove using
one description and very hard using another so we will switch between
them freely. The first one is very well known but, in most cases,
relatively hard to work with.

\begin{theorem}
  Suppose that $G$ is a network and $s$ and $t$ are vertices in
  $G$. For any oriented edge $xy$ in $G$ define $N(s,t,x,y)$ to be the
  number of spanning trees of $G$ that contain $xy$ (in that order) on
  their unique $st$ path and let $N$ be the total number of spanning
  trees in $G$.

  Then the flow in the edge $xy$ in the unique current flow with
  current of size 1 entering at $s$ and leaving at $t$ is given by
  \[\frac{N(s,t,x,y)-N(s,t,y,x)}{N}.
\]
\end{theorem}

The description of the unique current flow that we will use most
frequently is based on the idea of the power dissipated by a flow.
\begin{defn}
  Suppose that $G$ is a network and $s,t$ are vertices of $G$. For
  any flow from $s$ to $t$ (i.e. an assignment of currents that
  satisfies Kirchhoff's Current Law) define the \emph{power of the flow} to be
  \[\sum_{e\in G}I_e^2,
  \]
  where $I_e$ is the current in the edge $e$. 
\end{defn}
We remark that, formally, we have defined the current in an oriented
edge but, since the sign is irrelevant in this formula, it is
unambiguous.
\begin{theorem}\label{t:energy}
  Suppose that $G$ is a network and $s,t$ are vertices of
  $G$. Then the unit current flow from $s$ to $t$ is the unit flow
  from $s$ to $t$ that minimises the power, and the power of this flow
  is the effective resistance from $s$ to $t$.
\end{theorem}

We will make frequent use of Theorem~\ref{t:energy}; in many cases we
give an upper bound for the effective resistance by constructing a
flow with low power. For example the following is an easy consequence
of the power formulation, but relatively hard to prove from the
spanning tree formulation.

\begin{lemma}[Monotonicity]
  Suppose that $G$ is a network and that $G'=G-\{e\}$ for some $e\in
  G$. Then the effective resistance between any two vertices in $G$ is
  at most the effective resistance between them in $G'$.

  Similarly, if $G''$ is formed by identifying two vertices of $G$
  then the effective resistance between any two vertices in $G''$ is
  at most the effective resistance between them in $G$.
\end{lemma}
We prove the first part as a simple example of this argument.
\begin{proof}
  Let $x,y\in V(G)$ and let $I$ be the unit current flow in $G-\{e\}$
  from $x$ to $y$, so the power of the flow $I$ is $R_{xy}(G-\{e\})$.
  This is a unit flow in $G$ (i.e., with current 0 in edge $e$), and
  the power it dissipates in $G$ is the same as in $G'$. Thus the
  power of the flow that minimises the power is at most this much:
  i.e., $R_{xy}(G)\le R_{xy}(G')$.
\end{proof}

Finally, there is a linear algebraic interpretation of effective resistance involving eigenvalues. We define the Laplacian matrix $L$ of the graph $G$ to be the $n\times n$ matrix with entries:
\[
L_{ij}=\left\{\begin{array}{cl}
-1 & \text{ if $i\not=j$ and $ij\in E(G)$}\\
0 & \text{ if $i\not=j$ and $ij\not\in E(G)$}\\
\deg(i) & \text{ if $i=j$}
\end{array}\right.
\]
Since the row sums of $L$ are $0$, the matrix $L$ has $0$ as an eigenvalue. Provided that $G$ is connected, all the other eigenvalues $\lambda_1,\dots,\lambda_{n-1}$ are positive and it can be shown that
the total resistance is:
\[
\sum_{x<y} R_{xy}=2 \sum_{i=1}^{n-1} \lambda_{i}^{-1}
\]
So minimising $A(G)$ corresponds to minimsing the sum of the reciprocals of the non-trivial Laplacian eigenvalues.

In fact, the individual effective resistances can be expressed in terms of entries of the so-called Moore-Penrose generalised inverse of $L$ but we will not need this. See \cite{KleinRandic} and \cite{MR3026766} for further details.

\section*{Appendix B: Proof of Theorem~\ref{t:qb-never-best}}
We prove the bound claimed at the end of the proof of
Theorem~\ref{t:qb-never-best}. This is elementary but tedious, and the
calculations given were done with the aid of a computer algebra
package.

Define the function $l(x)=\frac{x-1}{x(x-2)}$ and $\Delta
l(x)=l(x+1)-l(x)$ (which will be negative). Then, for any integer $t$,
the line through the points $(t,l(t))$ and $(t+1,l(t+1))$ has equation 
\[y=y(x)=l(t)+(x-t)\Delta l(t).\] This is the equation of our upper
bound construction (the convex combination of $t$-regular and
$(t+1)$-regular graphs) for $x$ between $t$ and $t+1$. We need to show
that, in this range, this function is less than the Queen-Bee lower
bound of $1/(x-3/2)$. In fact we show, that this function is less than
$1/(x-3/2)$ for all $x>3/2$. Since $1/(x-3/2)$ is continuous and
greater than $y(x)$ for $x$ near $3/2$ it suffices to show that the
equation $1/(x-3/2)=y(x)$ has no solutions. This equation rearranges to
\[\Delta l(t)(x-t)(x-3/2)+(x-3/2)l(t)-1 =0
\]
which expands to 
\[\Delta l(t)x^2+\left(l(t)-dl(t)(t+3/2)\right)x+\frac{3t\Delta
  l(t)}2-\frac{3l(t)}2-1=0.
\]
Finding the discriminant of the quadratic we get
\[
-\frac {8\,{t}^{5}-41\,{t}^{4}+66\,{t}^{3}-71\,{t}^{2}+38\,t-1}{ 4\left( t+1 \right) ^{2
} \left( t-1 \right) ^{2}{t}^{2} \left( t-2 \right) ^{2}}.
\]
Writing $t=4+s$ this becomes
\[
-{\frac {8\,{s}^{5}+119\,{s}^{4}+690\,{s}^{3}+1905\,{s}^{2}+2382\,s+935}{ 4\left( s+5
 \right) ^{2} \left( s+3 \right) ^{2} \left( s+4 \right) ^{2} \left( s+2 \right) ^{2}}}
\]
which is obviously negative for all $s\ge 0$, i.e., for all $t\ge4$.

\medskip

The next case is when $\a\in[3\frac12,4]$. We need to use the convex
combination of our construction for average degree $10/3$ and the four
regular construction. For all $x\le 4$ this lies below the line with
equation
\[y(x)=(x-10/3)\frac{3/8-0.528}{2/3}+0.528.
\]
As above we look for solutions to $y(x)=1/(x-3/2)$. This equation rearranges to
\[
- 0.2295x^{2}+ 1.63725x- 2.9395=0
\]
which has discriminant less than 0. Thus there are no roots, and this
case is complete.

Next we need to consider the comparison to the other lower bound for
Queen-Bee networks: namely $(5-\a)/3$. For $\a<10/3$ this is trivially
above the line through the star $(2,1)$, and our construction for
$10/3$ (the point $(10/3,0.528)$). 

Finally we have to compare the bound of $(5-x)/3$ and the convex
combination of our construction for $10/3$ and the regular graph for
$x\le3\frac12$. Writing $y(x)$ for  this bound we get
\[\frac{5-x}{3}-y(x)= 0.373666667 - 0.1038333333 x
\]
which is zero at approximately $3.59$, so positive before that. This completes the proof.
\end{document}